\documentclass[a4paper,11pt,reqno]{amsart}
\usepackage{amsthm}
\usepackage{amsmath}
\usepackage{enumerate}
\usepackage{amsfonts}
\usepackage{amssymb}
\usepackage{fullpage}
\usepackage{amsmath,amscd}
\usepackage{stmaryrd}
\usepackage{graphicx}
\usepackage{array}
\usepackage{amsmath,amssymb,amsfonts,dsfont}
\usepackage[utf8]{inputenc} 
\usepackage[english]{babel}
\usepackage[T1]{fontenc} 
\usepackage{graphicx}
\usepackage{float}
\usepackage{pdfpages}
\usepackage[a4paper, margin = 3cm, bottom = 3cm]{geometry}
\usepackage{ifpdf}
\usepackage{marginnote}
\usepackage{hyperref}	
\hypersetup{pdfborder=0 0 0, 
	    colorlinks=true,
	    citecolor=black,
	    linkcolor=blue,
	    urlcolor=red,
	    pdfauthor={Guillaume Tahar}
	   }
\newtheorem{thm}{Theorem}[section]

\newtheorem{prop}[thm]{Proposition}
\newtheorem{lem}[thm]{Lemma}
\theoremstyle{definition}
\newtheorem{defn}[thm]{Definition}
\theoremstyle{remark}

\theoremstyle{definition}

\theoremstyle{definition}
\newtheorem{ex}[thm]{Example}
\theoremstyle{definition}

\numberwithin{equation}{section} 
\title{Veech groups of flat surfaces with poles}
\author{Guillaume Tahar} 
\address[Guillaume Tahar]{Institut de Math{\'e}matiques de Jussieu - UMR CNRS 7586}
\email{guillaume.tahar@imj-prg.fr}
\date{November 26, 2017}
\begin{document}
\begin{abstract}
Flat surfaces that correspond to meromorphic $1$-forms with poles or to meromorphic quadratic differentials containing poles of order two and higher have infinite flat area. We classify groups that appear as Veech groups of translation surfaces with poles. We characterize those surfaces such that their $GL^{+}(2,\mathbb{R})$-orbit or their $SL(2,\mathbb{R})$-orbit is closed. Finally, we provide a way to determine the Veech group for a typical infinite surface in any given chamber of a stratum.\newline
\end{abstract}
\maketitle
\setcounter{tocdepth}{1}
\tableofcontents

\section{Introduction}

Moduli spaces of translation structures on surfaces are endowed with an action of $GL^{+}(2,\mathbb{R})$. The Veech group of a flat surface $X$ is the stabilizer of $X$ by this group action. Characterization of Veech groups of translation surfaces corresponding to Abelian differentials is an open problem \cite{EM,Sc,Zo}. We show in this paper that flat surfaces corresponding to meromorphic differentials are rigid enough to allow a full solution to this problem.\newline
Some results about the Veech group of a flat surface of infinite area are already known. In \cite{Va}, Valdez computes the Veech group of some families of surfaces of infinite area associated to irrational billards. In \cite{HS}, Hubert and Schmithüsen provide examples of square tiled surfaces with an infinite number of tiles and whose Veech groups are infinitely generated subgroups of $SL(2,\mathbb{R})$. In this paper, we consider only surfaces of finite type, that is surfaces of finite genus and whose singularities have a finite order.\newline

Meromorphic $1$-forms and quadratic differentials with poles of higher order induce flat structures of infinite area. However, they have finite complexity and we have some tools to study them \cite{Bo1, HKK, Ta}. Since we cannot normalize the area, the action of $GL^{+}(2,\mathbb{R})$ does not reduce to that of $SL(2,\mathbb{R})$.\newline

Let $X$ be a compact Riemann Surface and $q$ be a meromorphic quadratic differential with at least one pole of higher order, that is a pole of order no smaller than two. Such a pair $(X,q)$ is called a \textit{flat surface with poles of higher order}. When there is no ambiguity, it will be referred to as $X$. The meromorphic quadratic differential $q$ is allowed to be the global square of a meromorphic differential so our study includes meromorphic $1$-forms as a special case.\newline

Zeroes and poles of order one of the quadratic differential are conical singularities for the induced flat metric \cite{St}. We denote by $\mathcal{Q}(a_1,\dots,a_n,-b_1,\dots,-b_p)$ the stratum in moduli space of meromorphic quadratic differentials that corresponds to meromorphic differentials with conical singularites of degrees $a_1,\dots,a_n \in \lbrace-1\rbrace \cup \mathbb{N}^{\ast}$ and poles of degrees $b_1,\dots,b_p \geq 2$. We have $\sum \limits_{i=1}^n a_i - \sum \limits_{j=1}^p b_j = 4g-4$ where $g$ is the \textit{genus} of the underlying Riemann surface. Unless otherwise indicated, we will assume that $n>0$ and $p>0$. Besides, if $g=0$, we will assume that $n+p \geq 3$. In order to simplify our study, there are no marked points.\newline
In this paper, we consider the Veech group of a flat surface with poles of higher order $(X,q)$ as a subgroup of $GL^{+}(2,\mathbb{R})/ \lbrace\pm Id\rbrace$. When the quadratic differential is the global square of a meromorphic $1$-form $\omega$, what we get is a finite index subgroup of the Veech group of $(X,\omega)$ using the usual definition.\newline

\section{Statement of main results}

There is a fundamental geometric invariant of translation surfaces with poles and their quadratic counterparts. The convex hull $core(X)$ of the set of conical singularities of a surface $X$ is a kind of polygon, see \cite{HKK,Ta}, whose boundary is a finite chain of saddle connections. The core is invariant for the action of the Veech group. Therefore, symmetries of the surface are significantly restricted. For instance, in most cases, the $GL^{+}(2,\mathbb{R})$ action is not ergodic \cite{Bo1}.\newline
Besides, unlike strata of translation surfaces, strata of flat surfaces with poles of higher order have a walls-and-chambers structure defined by topological changes of the core. The cores of two surfaces that belong to the same chamber are topologically identical.\newline

Our main theorem is a complete classification of Veech groups of flat surfaces with poles of higher order. Valdez proved an analogous classification in the context of Abelian differentials on stable curves, see \cite{Va}.\newline

\begin{thm} The Veech group of a flat surface with poles of higher order belongs to one of the following three types of subgroups of $GL^{+}(2,\mathbb{R})/ \lbrace\pm Id\rbrace$:\newline
- Finite type: conjugated to a finite rotation group ;\newline
- Cyclic parabolic type: conjugated to $\left\lbrace \begin{pmatrix} 1&k \\ 0 & 1 \end{pmatrix} \mid k \in \mathbb{Z} \right\rbrace $ ;\newline
- Continuous type: conjugated to $\left\lbrace\begin{pmatrix} 1 & b \\ 0 & a \end{pmatrix} \mid a \in \mathbb{R}^{*}_{+}, b \in \mathbb{R} \right\rbrace$.
\end{thm}

It should be noted that flat cones that belong to strata $\mathcal{Q}(a,-a-4)$ with $a \in \lbrace-1\rbrace \cup \mathbb{N}^{\ast}$ have the whole $GL^{+}(2,\mathbb{R})/ \lbrace\pm Id\rbrace$ as Veech group. We exclude these trivial cases by assuming $n+p \geq 3$.\newline

Theorem 2.1 is proved in section 4.\newline

For classical translation surfaces, the $GL^{+}(2,\mathbb{R})$-orbit is closed if and only if the Veech group is a lattice subgroup in $SL(2,\mathbb{R})$, see \cite{Ve}. There is an analogous result for translation surfaces with poles.

\begin{thm} The following statements are equivalent for a flat surface with poles of higher order $(X,q)$:\newline
(i) The Veech group of $(X,q)$ is of continuous type.\newline
(ii) All saddle connections of $(X,q)$ share the same direction.\newline
(iii) The $GL^{+}(2,\mathbb{R})$-orbit of $(X,q)$ is closed in the ambient stratum.
\end{thm}

As we cannot normalize the area of the surface, we have to distinguish the action of $GL^{+}(2,\mathbb{R})$ and that of $SL(2,\mathbb{R})$.\newline
Surfaces whose $SL(2,\mathbb{R})$-orbit is closed are very specific among classical translation surfaces. They are called \textit{Veech surfaces}. On the contrary, it is common for a translation surface with poles to have a closed $SL(2,\mathbb{R})$-orbit.\newline

\begin{thm} The $SL(2,\mathbb{R})$-orbit of a flat surface with poles of higher order $(X,q)$ is closed in the ambient stratum in these two cases:\newline
(i) Not all saddle connections of $\partial\mathcal{C}(X)$ share the same direction.\newline
(ii) All saddle connections of $\partial\mathcal{C}(X)$ share the same direction and $core(X)$ decomposes into a (maybe empty) family of cylinders of commensurable moduli.\newline
Otherwise, the $SL(2,\mathbb{R})$-orbit of $(X,q)$ is not closed in the ambient stratum.
\end{thm}

Theorems 2.2 and 2.3 are proved in section 5.\newline

Following definitions of \cite{Mo}, we say that a flat surface $(X,q)$ is generic in its stratum (or its chamber) if it lies outside a countable union of real codimension one submanifolds. Two saddle connections are said to be parallel when their relative homology classes are not linearly independant over $\mathbb{R}$.\newline

Generic Veech groups of flat surfaces of finite area are already known, see \cite{Mo}. They are trivial as long as $g\geq2$. We get similar results for flat surfaces with poles of higher order. It should be recalled that the dimension of a stratum $\mathcal{Q}$ is $2g+n+p-2$, see \cite{BCGGM1}.\newline

\begin{thm}
In strata of dimension one, the Veech group of every surface is of continuous type.\newline
In strata of dimension at least three, the Veech group of a generic surface is trivial.\newline
In strata of dimension two, we distinguish several cases:\newline
(i) The Veech group of every surface is of cyclic parabolic type in chambers where the core is a cylinder in the following strata:\newline
- $\mathcal{Q}(a,-a)$ with $a\geq2$,\newline
- $\mathcal{Q}(a,-1^{2},-2-a)$ with $a\geq1$,\newline
- $\mathcal{Q}(a,b,-a-2,-b-2)$ with $a,b\geq1$.\newline
(ii) In chambers where the core is a triangle, the Veech group of every surface is isomorphic to $\mathbb{Z}/3\mathbb{Z}$. Such chambers exist in the following strata:\newline
- $\mathcal{Q}(3b-4,-b,-b,-b)$ with $b\geq2$,\newline
- $\mathcal{Q}(a,a,a,-3a-4)$ with $a\geq 1$,\newline
(iii) In the chamber of each stratum $\mathcal{Q}(3a,-3a)$ with $a\geq1$ where the core is a triangle and the three elementary loops have the same topological index,
the Veech group of every surface is isomorphic to $\mathbb{Z}/3\mathbb{Z}$.\newline
(iv) The Veech group of every surface is isomorphic to $\mathbb{Z}/2\mathbb{Z}$ in the the open $GL^{+}(2,\mathbb{R})$-orbits of the following flat surfaces:\newline
- Flat surface whose core is degenerated and where all angles are congruent angles in $\mathcal{Q}(4k,-4k)$ with $k\geq1$,\newline
- Flat surface whose core is degenerated and whose four angles are $(2+a)\pi$, $(2+a)\pi$, $\dfrac{(2k+1)\pi}{2}$ and $\dfrac{(2k+1)\pi}{2}$ in strata $\mathcal{Q}(2k-1,a,a,-2a-2k-3)$ with $k\geq1$ and $a \in \lbrace-1\rbrace \cup \mathbb{N}^{\ast}$,\newline
- Flat surface whose core is degenerated and whose four angles are $(b-1)\pi$, $(b-1)\pi$, $\dfrac{(2k+1)\pi}{2}$ and $\dfrac{(2k+1)\pi}{2}$ in $\mathcal{Q}(2b+2k-3,-b,-b,-2k-1)$ with $b\geq2$ and $k\geq1$,\newline
- Flat surface whose core is formed by two saddle connections between two distinct conical singularities and whose four angles are $\dfrac{(1+2k)\pi}{2}$, $\dfrac{(1+2k)\pi}{2}$, $\dfrac{(1+2l)\pi}{2}$ and $\dfrac{(1+2l)\pi}{2}$ in $\mathcal{Q}(k+l-1,k+l-1,-2k-1,-2l-1)$ with $k,l \geq 1$,\newline
- Flat surface whose core is formed by two saddle connections between two distinct conical singularities and whose four angles are $\dfrac{(1+2k)\pi}{2}$, $\dfrac{(1+2k)\pi}{2}$, $\dfrac{(1+2l)\pi}{2}$ and $\dfrac{(1+2l)\pi}{2}$ in $\mathcal{Q}(2k-1,2l-1,-k-l-1,-k-l-1)$ with $k,l \geq 1$.\newline
(v) Outside these chambers or orbits, the Veech group of a generic surface is trivial.
\end{thm}

Theorem 2.4 is proved in section 6.\newline

The structure of the paper is the following: \newline
- In Section 3, we recall the background and tools useful to study flat surfaces of infinite area: flat metric, saddle connections, the moduli space, the core of a surface and its associated wall-and-chambers structure, the contraction flow.\newline
- In Section 4, we prove our theorem of classification.\newline
- In Section 5, we provide complete criteria of closedness of orbits of $GL^{+}(2,\mathbb{R})$ and $SL(2,\mathbb{R})$.\newline
- In Section 6, we characterize loci where the Veech group of a generic surface is nontrivial.\newline

\textit{Acknowledgements.} I thank my doctoral advisor Anton Zorich for motivating my work on this paper and many interesting discussions. I am grateful to Corentin Boissy, Quentin Gendron, Ben-Michael Kohli and the reviewer for their valuable remarks.
This work is supported by the ERC Project "Quasiperiodic" of Artur Avila.\newline

\section{Definitions and tools}

\subsection{Flat structures}

Let $X$ be a compact Riemann surface and let $q$ be a meromorphic quadratic differential.
$\Lambda$ and $\Delta$ respectively are the set of conical singularities and poles of higher order of $q$. Outside $\Lambda$ and $\Delta$, $q$ is locally the square of a holomorphic differential $\omega$. Integration of $\omega$ in a neighborhood of $z_{0}$ gives local coordinates whose transition maps are of the type $z \mapsto \pm z+c$. The pair $(X,q)$ seen as a smooth surface with such a translation atlas is called a \textit{flat surface with poles of higher order}.\newline

In the case of quadratic differentials that are the global square of a $1$-form, the holonomy of the metric is trivial. Otherwise, the holonomy of the metric is $\mathbb{Z}/2\mathbb{Z}$.\newline

In a neighborhood of an element of $\Lambda$, the metric induced by $q$ admits a conical singularity of angle $\left(k+2 \right)\pi$ where $k$ is the degree of the corresponding zero of $q$. In particular, poles of order one are conical singularities of angle $\pi$.\newline

See Strebel \cite{St} for a complete description of the local geometry around poles of higher order.\newline

\begin{defn} A saddle connection is a geodesic segment joining two conical singularities of the flat surface such that all interior points are not conical singularities.
\end{defn}

In \cite{Ta}, we proved a bound on the number of saddle connections of a flat surface with poles of higher order.

\begin{thm}
Let $|SC|$ be the number of saddle connections of a flat surface with poles of higher order $(X,q)$ of genus $g$ belonging to $\mathcal{Q}(a_1,\dots,a_n,-b_1,\dots,-b_p)$, then we have $|SC| \geq 2g-2+n+p$.
\end{thm}

\subsection{Moduli space}

If $(X,q)$ and $(X',q')$ are flat surfaces such that there is a biholomorphism $f$ from $X$ to $X'$ such that $q$ is the pullback of $q'$, then $f$ is an isometry for the flat metrics defined by $q$ and $q'$.\newline

We define the moduli space of meromorphic quadratic differentials as the space of equivalence classes of pairs $(X,q)$ up to biholomorphism preserving the quadratic form.\newline

We denote by $\mathcal{Q}(a_1,\dots,a_n,-b_1,\dots,-b_p)$ the \textit{stratum} that corresponds to meromorphic quadratic forms with singularities of degrees $a_1,\dots,a_n$ and $-b_1,\dots,-b_p$. We have $a_1,\dots,a_n \in \lbrace-1\rbrace \cup \mathbb{N}^{\ast}$. They are the poles of order one and the zeroes of arbitrary order. They define conical singularities. Singularities are referred to as poles of higher order when they are poles of order no smaller than two. That is why we have $b_1,\dots,b_p \geq 2$.

\subsection{Canonical double covering}

Outside the set of singularities, every quadratic differential is locally the square of a $1$-form. The canonical double covering assigns to every pair $(X,q)$ of a Riemann surface and a quadratic differential a flat surface such that the quadratic differential is globally the square of a $1$-form, see \cite{BCGGM1}. 
For quadratic surfaces that are already the global square of a $1$-form $\omega$, instead of constructing a canonical double covering, we simply choose a square root.\newline

Let $(X,q)$ be a flat surface of $\mathcal{Q}(2c_{1},\dots,2c_{s},2d_{1}+1,\dots,2d_{t}+1)$ where $c_{1},\dots,c_{s}\in\mathbb{Z}^{\ast}$ and $d_{1},\dots,d_{t}\in\mathbb{Z}$. We denote by $P_{i} \in X$ the point corresponding to the singularity of order $2d_{i}+1$ and by $(\widetilde{X},\omega)$ the canonical double cover of $(X,q)$. Translation surface $(\widetilde{X},\omega)$ belongs to $\mathcal{H}(c_{1},\dots,c_{s},c_{1},\dots,c_{s},2d_{1}+2,\dots,2d_{t}+2)$.\newline

Quadratic differentials that are the global square of a $1$-form belong to strata of the form $\mathcal{Q}(2c_{1},\dots,2c_{s})$ where all singularities are of even order. Their square roots belong to stratum $\mathcal{H}(c_{1},\dots,c_{s})$.

\subsection{Period coordinates}

The double canonical covering $(\widetilde{X},\omega)$ of $(X,q)$ admits an involution $\tau$ that induces another involution $\tau^{\ast}$ on $H_{1}(X\setminus\Delta,\Lambda)$ where $\Delta$ is the set of poles of higher order and $\Lambda$ is the set of conical singularities.\newline
Thus $H_{1}(X\setminus\Delta,\Lambda)$ decomposes into an invariant subspace $H_{+}$ and an anti-invariant space $H_{-}$. Strata are complex-analytic orbifolds with local coordinates given by the period map of $H_{-}$, see Theorem 2.1 in \cite{BCGGM1}. In particular, its complex dimension is $2g+n+p-2$.\newline

We can associate homology classes to saddle connections. In the case of quadratic differentials, the way to do that is somewhat subtle, see \cite{MZ} for details. We consider a flat surface with poles $(X,q)$ and its double canonical covering $(\widetilde{X},\omega)$.

$H_{1}(\widetilde{X}\setminus\Delta',\Lambda')$ is the first relative homology group of $(\widetilde{X},\omega)$ where $\Delta'$ and $\Lambda'$ respectively are the preimages by $\pi$ of the conical singularities and poles of higher order associated to $q$. An involution $\tau$ associated to the covering acts on $H_{1}(\widetilde{X}\setminus\Delta',\Lambda')$.

We denote by $\gamma_{1}$ and $\gamma_{2}$ the two preimages of an oriented saddle connection $\gamma$. If the relative cycle $[\gamma_{1}]\in H_{1}(\widetilde{X}\setminus\Delta',\Lambda')$ satisfies $[\gamma_{1}]=- [\gamma_{2}]$, then we define $[\gamma]=[\gamma_{1}]$. Conversely, if $[\gamma_{1}]=[\gamma_{2}]$, then we define $[\gamma]=[\gamma_{1}]-[\gamma_{2}]$.\newline
Two saddle connections are said to be parallel when their relative homology classes are linearly dependant over $\mathbb{R}$.\newline
The holonomy vector of a saddle connection is essentially determined by the period of its relative homology class. Its direction modulo $\pi$ is that of the period and its length is one-half of the modulus of the period.

\subsection{Index of a loop}

The topological index of a loop is particularly easy to handle in flat geometry.\newline

\begin{defn}
Let $\gamma$ be a simple closed curve in a flat surface with (or without) poles of higher order. $\gamma$ is parametrized by arc-length and passes only through regular points.\newline
We consider the lifting $\eta$ of $\gamma$ by the canonical double covering. Then $\eta'(t)=e^{i\theta(t)}$.\newline
We have $\dfrac{1}{2\pi}\int_{0}^{T} \theta'(t)dt \in \dfrac{1}{2}\mathbb{Z}$ because of the holonomy of the flat surface. This number is the topological index $ind(\gamma)$ of our loop $\gamma$.\newline
\end{defn}

In particular, the topological index of a simple closed curve around a singularity of order $k$ is $1+\dfrac{k}{2}$.

\subsection{Core of a flat surface with poles of higher order}

The core of a flat surface of infinite area was introduced in \cite{HKK} and systematically used in \cite{Ta}. It is the convex hull of the conical singularities of a flat surface. Since all saddle connections belong to it, the core encompasses most of the qualitative (see walls-and chambers structure in subsection 3.7) and quantitative (periods of the homology) information about the geometry of the flat surface.

\begin{defn} A subset $E$ of a flat surface is \textit{convex} if and only if every element of any geodesic segment between two points of $E$ belongs to $E$.\newline
The core of a flat surface with poles of higher order $(X,q)$ is the convex hull $core(X)$ of its conical singularities $\Lambda$.\newline
$\mathcal{I}\mathcal{C}(X)$ is the interior of $core(X)$ in $X$ and $\partial\mathcal{C}(X) = core(X)\ \backslash\ \mathcal{I}\mathcal{C}(X)$ is its boundary.\newline
The \textit{core} is said to be degenerated when $\mathcal{I}\mathcal{C}(X)=\emptyset$ that is when $core(X)$ is just graph $\partial\mathcal{C}(X)$.
\end{defn}

\begin{lem} Let $X$ be a flat surface with poles of higher order $\mathcal{Q}(a_1,\dots,a_n,-b_1,\dots,-b_p)$, then $X \setminus core(X)$ has $p$ connected components. Each of them is a topological disk. We refer to these connected components as \textit{domains of poles}.
\end{lem}

\begin{proof}
Following Proposition 2.3 in \cite{HKK}, $core(X)$ is a deformation retract of $X \backslash \Delta$ where $\Delta$ is the set of poles of higher order.
\end{proof}

\begin{lem} For any flat surface with poles of higher order $X$, $\partial\mathcal{C}(X)$ is a finite union of saddle connections.
\end{lem}

\begin{proof} See Proposition 2.2 in \cite{HKK}.
\end{proof}

In \cite{Ta}, we proved an upper bound on the maximal number of noncrossing saddle connections. The bound depends on the number $\beta$ of saddle connections that belong to the boundary of the core (counted twice if the two sides of a saddle connection belong to domains of poles). A maximal graph of noncrossing saddle connections defines a flat triangulation of the core formed by ideal triangles (triangles where the vertices may be the same).

\begin{thm}
Let $|A|$ be the maximal number of noncrossing saddle connections of a flat surface with poles of higher order $(X,q)$ of genus $g$ belonging to $\mathcal{Q}(a_1,\dots,a_n,-b_1,\dots,-b_p)$, then we have $|A| = 6g-6+3n+3p-\beta$. Besides,
$|A|=2g-2+n+p+t$ where $t$ is the number of ideal triangles in any flat triangulation of the core.
\end{thm}

\subsection{Discriminant and walls-and-chambers structure}

Strata of flat surfaces with poles of higher order decompose into chambers where the qualitative shape of the core is the same. The discriminant is the locus that separates these chambers from each other.

\begin{defn} A flat surface with poles of higher order $X$ belongs to the discriminant of the stratum if and only if there exists a pair of two nonparallel consecutive saddle connections of the boundary of the core that share an angle of $\pi$. Chambers are defined to be the connected components of the complementary to the discriminant in the strata.
\end{defn}

The following lemma is proved as Proposition 4.12 in \cite{Ta}.

\begin{lem} The discriminant is a $GL^{+}(2,\mathbb{R})$-invariant hypersurface of real codimension one in the stratum.
\end{lem}

The topological map on a flat surface with poles of higher order $(X,q)$ defined by the embedded graph $\partial\mathcal{C}(X)$ is invariant along the chambers. The qualitative shape of the core and in particular the number of saddle connections of its boundary depend only on the chamber (see Proposition 4.13 in \cite{Ta} for details).

\subsection{Dynamics and decomposition into invariant components}

We essentially follow the definitions Strebel gives in \cite{St}.

\begin{defn} Depending on the direction, a trajectory  starting from a regular point is of one of the four following types:\newline
- regular closed geodesic (the trajectory is periodic),\newline
- critical trajectory (the trajectory reaches a conical singularity in finite time),\newline
- trajectory finishing at a pole (the trajectory converges to a pole of higher order as $t\rightarrow+\infty$),\newline
- recurrent trajectory (infinite trajectory nonconverging to a pole of higher order).
\end{defn}

Theorem 3.11 describes how the directional flow decomposes flat surfaces into a finite number of invariant components. This theorem is proved as Theorem 2.3 in \cite{Ta}.

\begin{thm} Let $(X,q)$ be a flat surface with poles of higher order. Cutting along all saddle connections sharing a given direction $\theta$, we obtain finitely many connected components called \textit{invariant components}. 
There are four types of invariant components: \newline
- \textbf{finite volume cylinders} where the leaves are periodic with the same period, \newline
- \textbf{minimal components} of finite volume where the foliation is minimal, the directions are recurrent and whose dynamics are given by a nontrivial interval exchange map, \newline
- \textbf{infinite volume cylinders} bounding a simple pole and where the leaves are periodic with the same period, \newline
- \textbf{free components} of infinite volume where generic leaves go from a pole to another or return to the same pole.\newline
Finite volume components belong to $core(X)$.
\end{thm}

\subsection{$GL^{+}(2,\mathbb{R})$ action and contraction flow}

On each stratum, $GL^{+}(2,\mathbb{R})$ acts by composition with coordinate functions, see \cite{Zo}. Since neighborhoods of poles of higher order have infinite area, we cannot normalize the area of the surface and thus must consider the full action of $GL^{+}(2,\mathbb{R})$.\newline

For a flat surface $(X,q)$, the \textit{stabilizer} $stab(X) \subset GL^{+}(2,\mathbb{R})$ is the subgroup of those $g \in GL^{+}(2,\mathbb{R})$ for which $gX = X$. The quotient $stab(X)/\lbrace\pm Id\rbrace \subset GL^{+}(2,\mathbb{R})/\lbrace\pm Id\rbrace$ is called the \textit{Veech Group} of the flat surface $(X,q)$.\newline
Veech groups of two surfaces belonging to the same $GL^{+}(2,\mathbb{R})$-orbit are conjugated.
\newline

The contraction flow is a tool that allows to construct surfaces with degenerated core in a systematic way. All foundation results are already proved in \cite{Ta}.\newline

\begin{defn} Let $\alpha$ and $\theta$ be two distinct directions. The contraction flow is the action of the semigroup of matrices $C^{t}_{\alpha,\theta}$ conjugated to $\begin{pmatrix} e^{-t} & 0 \\ 0 & 1 \end{pmatrix}$ such that $\alpha$ is the contracted direction and $\theta$ is the preserved direction.
\end{defn}

\begin{lem}
Let $(X,q)$ be a flat surface with poles of higher order. $\theta$ is a direction that is not a direction of a saddle connection and $\alpha$ is any direction different from $\theta$. The sequence $C^{t}_{\alpha,\theta}(X,q)$ converges to a surface $(X_{0},q_{0})$ in the ambiant stratum. All saddle connections of $(X_{0},q_{0})$ share the same direction $\theta$ and $core(X_{0})$ is degenerated. 
\end{lem}

\begin{proof}
See Lemma 2.2 of \cite{Ta}.
\end{proof}

\section{Classification of Veech Groups}

Veech groups of flat surfaces with poles of higher order are very different from those of usual translation surfaces. They are either very big (not discrete) or very small (virtually cyclic). For instance, they cannot be lattices so there are no Veech surfaces with poles.\newline

\begin{proof}[Proof of Theorem 2.1]
Let $(X,q)$ be a flat surface with poles of higher order. There is a finite number of saddle connections in $\partial\mathcal{C}(X)$. The holonomy vectors of these saddle connections are a finite subset of $\mathbb{C}/\lbrace\pm1\rbrace$. Since this set is preserved by the action of the Veech group, either all saddle connections of $\partial\mathcal{C}(X)$ share the same direction or the Veech group is a finite group. Finite subgroups of $GL^{+}(2,\mathbb{R})$ are rotation groups. We say in this case that the Veech group is of finite type.\newline

In the following, we consider that all saddle connections of $\partial\mathcal{C}(X)$ share the same direction $\theta$. This direction is unchanged by the action of the Veech group and the lengths of these saddle connections are preserved. Therefore, in this case, the Veech group is conjugated to a subgroup of $\left\lbrace\begin{pmatrix} 1 & b \\ 0 & a \end{pmatrix} \mid a,b \in \mathbb{R} \right\rbrace$.\newline

The next disjonction is between surfaces with degenerated and nondegenerated core. If $core(X)$ is degenerated, then every saddle connection belongs to $\partial\mathcal{C}(X)$ and by hypothesis they belong to the same direction. Without loss of generality, we can consider they all belong to the horizontal direction. The whole geometry of $X$ is that of infinite vertical strips whose intersection with $core(X)$ are horizontal saddle connections. Thus, the action of every element of $GL^{+}(2,\mathbb{R})$ preserving the horizontal direction leaves $(X,q)$ unchanged. Thus, the Veech group of $(X,q)$ is conjugated to $\left\lbrace\begin{pmatrix} 1 & b \\ 0 & a \end{pmatrix} \mid a,b \in \mathbb{R} \right\rbrace$.\newline

If $core(X)$ is not degenerated, then the Veech group of $(X,q)$ is a subgroup of 
$SL(2,\mathbb{R})/ \lbrace\pm Id\rbrace$ because its action preserves the area of the core. Besides, nondegenerated core implies that there are two saddle connections that have different directions. The set of holonomy vectors of the saddle connections of the whole surface is a discrete subset $S$ of $\mathbb{C}/\lbrace\pm1\rbrace$ that is preserved by the action of the Veech group. As there are two saddle connections that have different directions, then an element of the Veech group is entirely characterized by its action on $S$. Therefore, the Veech group of $X$ is a discrete subgroup of $GL^{+}(2,\mathbb{R})/ \lbrace\pm Id\rbrace$. Besides, the action of the Veech group preserves the area of the core. This implies in particular that the Veech group of $(X,q)$ is a discrete subgroup of $SL(2,\mathbb{R})/ \lbrace\pm Id\rbrace$. We know it is also conjugated to a subgroup of $\left\lbrace\begin{pmatrix} 1 & b \\ 0 & a \end{pmatrix} \mid a,b \in \mathbb{R} \right\rbrace$. Therefore, either it is conjugated to $\left\lbrace \begin{pmatrix} 1&k \\ 0 & 1 \end{pmatrix} \mid k \in \mathbb{Z} \right\rbrace$ or it is the trivial group.\newline
\end{proof}

One can find a realization for every Veech group of the classification provided by Theorem 2.1. The following example shows in particular that every cyclic finite group is realizable as a Veech group.

\begin{ex} Gluing infinite cylinders on the edges of a regular $2k$-gon, we get a surface in $\mathcal{Q}(4k-4,-2^{2k})$ whose Veech group is conjugated to $\mathbb{Z}/k\mathbb{Z}$, see Figure 1.

\begin{figure}
\includegraphics[scale=0.3]{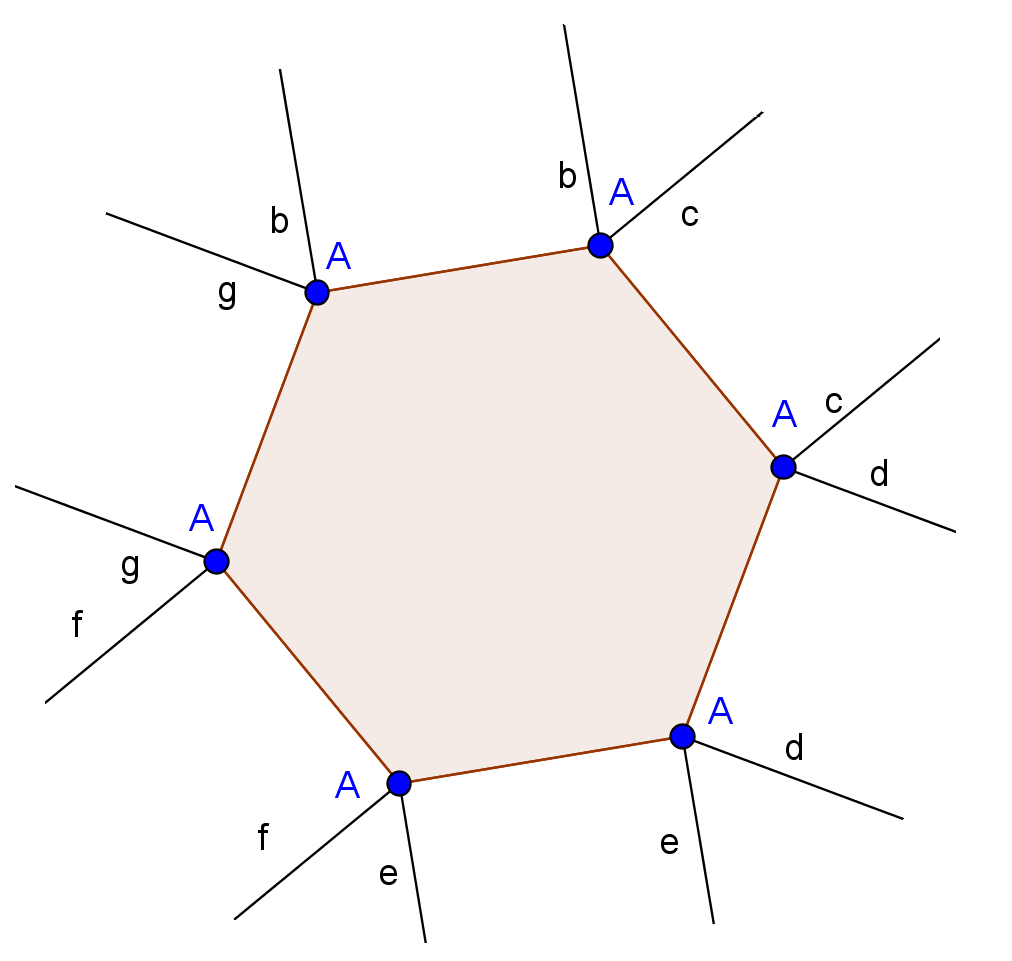}
\caption{A flat surface in $\mathcal{Q}(8,-2^{6})$ with a discrete rotational symmetry.}
\end{figure}
\end{ex}

In the following, we characterize surfaces with Veech groups of cyclic parabolic type in terms of the decomposition in invariant components.

\begin{prop} The Veech group of a flat surface with poles of higher order $X$ is of cyclic parabolic type if and only if the following proposition hold:\newline
- every saddle connection of $\partial\mathcal{C}(X)$ belongs to the same direction $\theta$,\newline
- $core(X)$ is not degenerated and admits a decomposition into a finite number of cylinders with commensurable moduli in direction $\theta$.\newline
\end{prop}

\begin{proof}
We first assume the Veech group of a flat surface $(X,q)$ is of cyclic parabolic type. We denote by $\theta$ the direction preserved by the parabolic element $M$. Since the boundary of $core(X)$ is a finite union of saddle connection that is globally preserved by the action of the parabolic element, all of them share the same direction $\theta$. Besides, $core(X)$ is not degenerate because otherwise the Veech group would be of continuous type, see proof of Theorem 2.1.\newline

Among flat surfaces whose core is not degenerate and such that every saddle connection of $\partial\mathcal{C}(X)$ belongs to the same direction $\theta$, we will prove that there is a parabolic element in the Veech group that preserves direction $\theta$ if and only if $core(X)$ decomposes into a finite number of cylinders with commensurable moduli in direction $\theta$.\newline

We identify the saddle connections in the boundary of the interior of the core $\mathcal{I}\mathcal{C}(X)$ in an arbitrary way and get a classical half-translation  surface of finite area $(Y,q')$ whose dynamics in direction $\theta$ is exactly the same as that in $\mathcal{I}\mathcal{C}(X)$. It must be noted that $Y$ may fail to be connected.\newline
It is well known that for classical translation surfaces (and half-translation surfaces), existence of a parabolic element of the Veech group preserving a direction $\theta$ is equivalent to existence of a decomposition of the surface into a finite number of cylinders whose closed geodesics belong to direction $\theta$ and whose moduli are commensurable, see Subsection 2.3 in \cite{HL} or \cite{Ve}.\newline
Consequently, if the Veech group of $(X,q)$ is of cyclic parabolic type, then there is a parabolic element $M$ that preserves $core(X)$ and therefore $(Y,q')$. Therefore, $(Y,q')$ (and $\mathcal{I}\mathcal{C}(X)$) admits a decomposition into cylinders whose moduli are commensurable.\newline
Conversely, if $\mathcal{I}\mathcal{C}(X)$) admits a decomposition into cylinders whose moduli are commensurable in some direction $\theta$, then $(Y,q')$ admits the same decomposition and there is a parabolic element that preserves $(Y,q')$ and direction $\theta$. Since the rest of $X$ is formed by saddle connections whose direction is also $\theta$, then parabolic element $M$ preserves the whole surface $(X,q)$ and its Veech group is of cyclic parabolic type.
\end{proof}

\begin{ex} We can construct many surfaces with a Veech group of cyclic parabolic type starting from a square-tiled core, see Figure 2.

\begin{figure}
\includegraphics[scale=0.3]{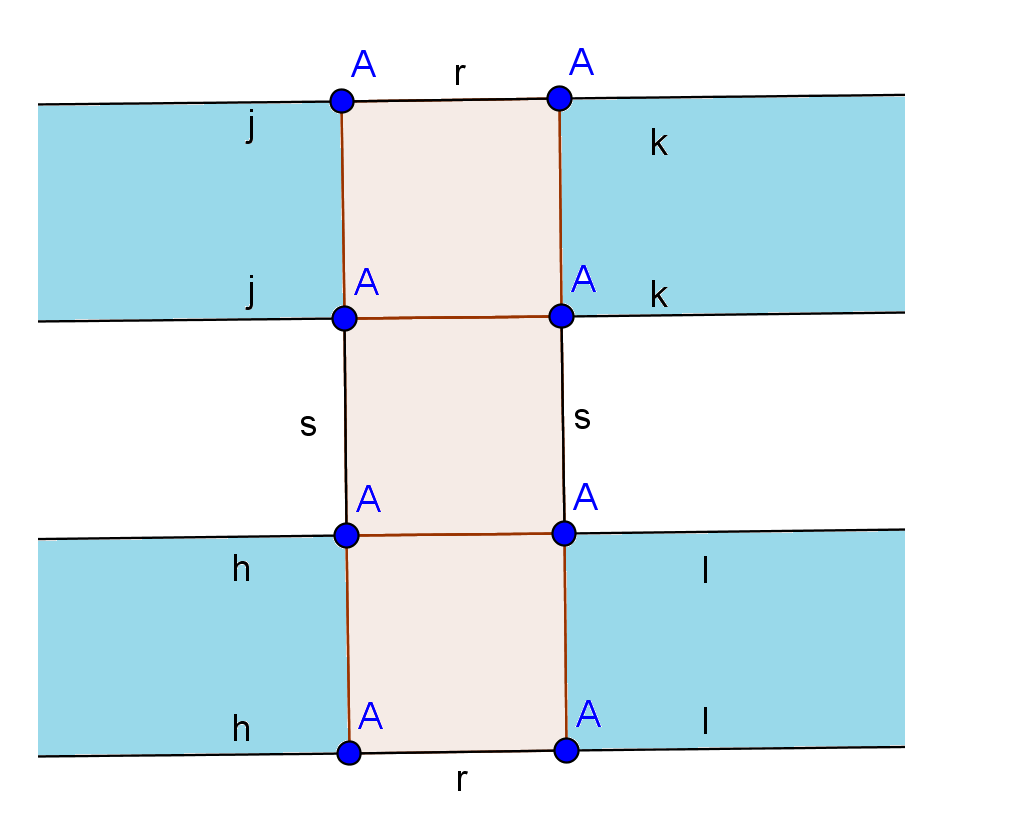}
\caption{A flat surface in $\mathcal{Q}(8,-2^{4})$ whose Veech group is of cyclic parabolic type.}
\end{figure}
\end{ex}

Starting from any surface, the contraction flow in a generic direction provides examples of surfaces with a Veech group of infinite type, see Figure 3.\newline

\begin{figure}
\includegraphics[scale=0.3]{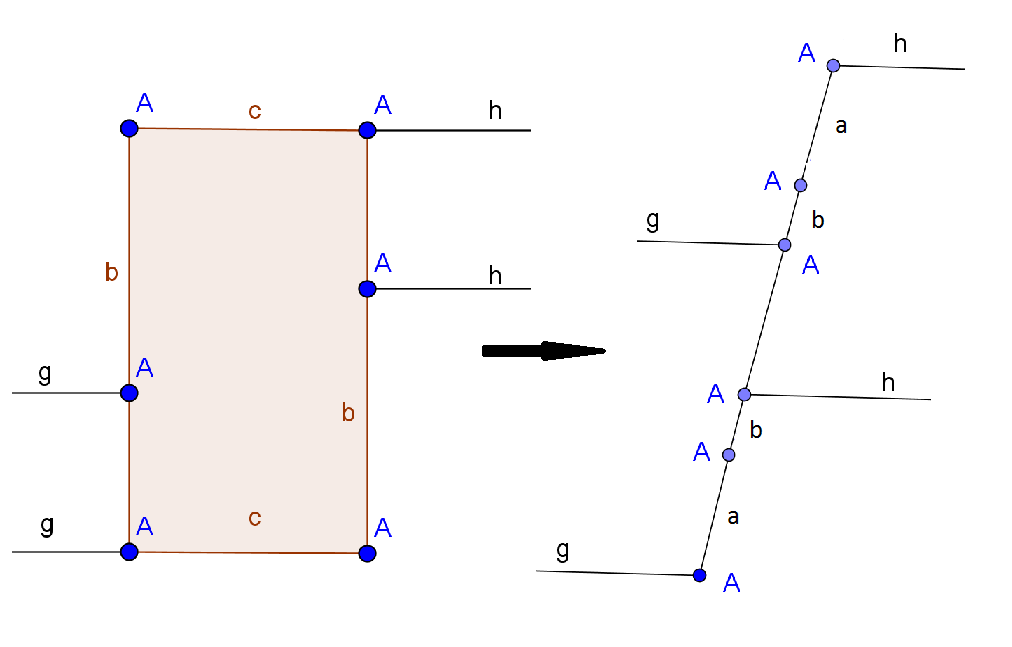}
\caption{A surface in $\mathcal{Q}(4,-2^{2})$ and the surface obtained through the use of the contraction flow along a given direction.}
\end{figure}

The contraction flow provides many examples of surfaces with a Veech group of infinite type. However, many of these surfaces belong to the discriminant of the ambiant stratum. The following example shows that this is not always true.

\begin{ex}
Three flat planes with a vertical slit in each of them and connected cyclically provide a surface in $\mathcal{Q}(4^{2},-4^{3})$ with a Veech group of continuous type and that does not belong to the discriminant.
\end{ex}

\section{Closedness of orbits}

In the study of classical translation surfaces, the Veech surfaces have a particular importance because they are the class of surfaces such that their $GL^{+}(2,\mathbb{R})$-orbit is closed, see \cite{HSc}. Here, we characterize the class of flat surfaces with poles of higher order that share such a property. Theorem 2.2 proves that surfaces whose $GL^{+}(2,\mathbb{R})$-orbit is closed are exactly the surfaces whose Veech group is of continuous type.

\begin{proof}[Proof of Theorem 2.2]
We first prove that (i) implies (iii). Let $G$ be the Veech group of continuous type of a flat surface $(X,q)$. $\mathbb{C}^{\ast} \cong \left\lbrace\begin{pmatrix} a\cos(\theta) & -a\sin(\theta) \\ a\sin(\theta) & a\cos(\theta) \end{pmatrix} \mid a \in \mathbb{R}^{\ast},\ \theta \in \mathbb{R} \right\rbrace$ and $G$ generate $GL^{+}(2,\mathbb{R})$. So the action of $\mathbb{C}^{\ast}$ by rotations and homotheties is an affine parametrization of $GL^{+}(2,\mathbb{R}).(X,q)$. Therefore, $GL^{+}(2,\mathbb{R}).(X,q)$ is an affine plane that is closed.\newline

Then we prove that (iii) implies (ii). If there are two saddle connections of different directions in a flat surface with poles of higher order $(X,q)$, pushing the surface $(X,q)$ through the contraction flow along a direction $\theta$ that is not a direction of saddle connection gives in the limit a surface $(X_{0},q_{0})$ whose saddle connections share the same direction $\theta$. So $(X_{0},\omega_{0}) \notin GL^{+}(2,\mathbb{R}).(X,q)$ which is not closed. We proved by contraposition that (iii) implies (ii).\newline

Finally, a flat surface $(X,q)$ such that all saddle connections share the same direction has a Veech group that is of continuous type, see the proof of Theorem 2.1 for details. Therefore (ii) implies (i). This ends our circular reasoning.\newline
\end{proof}

While surfaces with a closed $GL^{+}(2,\mathbb{R})$-orbit are quite exceptional, most flat surfaces with poles of higher order have a closed $SL(2,\mathbb{R})$-orbit. Theorem 2.3 characterizes the surfaces whose $SL(2,\mathbb{R})$-orbit is closed in the ambient stratum.

\begin{proof}[Proof of Theorem 2.3]
We consider a flat surface $(X,q)$ with poles of higher order such that there are two saddle connections of different directions in $\partial\mathcal{C}(X)$. As their holonomy vectors are defined with an ambiguity of $\pm1$, we consider one of the representatives in $\mathbb{C}$ for each of these two saddle connections. We denote by $w$ and $z$ the representatives. They form a basis of $\mathbb{C}$ as a $\mathbb{R}$-vector space. Every surface $(Y,q')$ in the $SL(2,\mathbb{R})$-orbit closure of $(X,q)$ arises as the limit of a sequence $A_{n}.(X,q)$ where $\forall n\geq0$, we have $A_{n} \in SL(2,\mathbb{R})$. As there are finitely many saddle connections in $\partial\mathcal{C}(Y)$, sequences $A_{n}.w$ and $A_{n}.z$ converge respectively to $z_{\infty}$ and $w_{\infty}$ in $\mathbb{C}$. Since matrices in $SL(2,\mathbb{R})$ are area-preserving, $(z_{\infty},w_{\infty})$ is a basis of $\mathbb{C}$ as well. A matrix of $SL(2,\mathbb{R})$ is characterized by the image of a basis. Thus, sequence $A_{n}$ converges to a limit $A_{\infty} \in SL(2,\mathbb{R})$. Therefore, $(Y,q')=A_{\infty}(X,q)$. In other words, every surface in the $SL(2,\mathbb{R})$-orbit closure of $(X,q)$ actually belongs to the orbit of $(X,q)$.\newline

If the Veech group of a flat surface with poles of higher order $(X,q)$ is of continuous type, then $GL^{+}(2,\mathbb{R}).(X,q)$ is closed (Theorem 2.2). As $SL(2,\mathbb{R})$ is a closed subgroup of $GL^{+}(2,\mathbb{R})$, $SL(2,\mathbb{R}).(X,q)$ is closed too. This covers situations where all saddle connections of $\partial\mathcal{C}(X)$ share the same direction and $core(X)$ is degenerated.\newline

Then we consider surfaces $(X,q)$ such that all saddle connections of $\partial\mathcal{C}(X)$ share the same direction $\theta$ and $core(X)$ is not degenerated. Without loss of generality, we suppose that $\theta$ is the horizontal direction. We are going to use the Iwasawa decomposition of $SL(2,\mathbb{R})$.\newline
Every $M \in SL(2,\mathbb{R})$ has a unique representation as $M=K.A.N$ where we have\newline
$K=\begin{pmatrix} 
\cos(\alpha) & -\sin(\alpha) \\
\sin(\alpha) & \cos(\alpha) 
\end{pmatrix}$,
$A=
\begin{pmatrix} 
r & 0 \\
0 & 1/r 
\end{pmatrix}$ and
$N=
\begin{pmatrix} 
1 & x \\
0 & 1 
\end{pmatrix}$ with $\alpha,x \in \mathbb{R}$ and $r>0$. In particular, $K \in SO(2,\mathbb{R})$.\newline

Every surface $(Y,q')$ in the $SL(2,\mathbb{R})$-orbit closure of $(X,q)$ arises as the limit of a sequence $M_{n}.(X,q)$ where $\forall n\geq0$, we have $M_{n} \in SL(2,\mathbb{R})$. Besides, $\forall n \geq0$, we have $M_{n}=K_{n}A_{n}N_{n}$ in the Iwasawa decomposition.\newline

First, $SO(2,\mathbb{R})$ is compact so there is $\alpha_{\infty} \in \mathbb{R}$ and an increasing subsequence $(\phi(n))_{n\geq0}$
such that we have:\newline
$$\lim\limits_{n \rightarrow +\infty} K_{\phi(n)}=
\begin{pmatrix} 
\cos(\alpha_{\infty}) & -\sin(\alpha_{\infty}) \\
\sin(\alpha_{\infty}) & \cos(\alpha_{\infty}) 
\end{pmatrix}
\in SO(2,\mathbb{R}).
$$\newline

We have
$A_{n}=
\begin{pmatrix} 
r_{n} & 0 \\
0 & 1/r_{n} 
\end{pmatrix}
$.
Since all saddle connections of $\partial\mathcal{C}(X)$ share the same direction, if there is an increasing subsequence $(\mu(n))_{n\geq0}$ such that $\lim\limits_{n \rightarrow +\infty} r_{\mu(n)}=+\infty$, then the length of the longest saddle connection in the boundary of the core of $M_{\mu(n)}.(X,q)$ grows to infinity. Therefore, $(r_{n})_{n\geq0}$ is bounded. For the same reasons (no saddle connection can shrink to saddle connection of zero length), $(r_{n})_{n\geq0}$ is bounded below by a positive number. That is why there is 
$r_{\infty}>0$ and an increasing subsequence $(\psi(n))_{n\geq0}$ such that $\lim\limits_{n \rightarrow +\infty} r_{\psi\circ\phi(n)}=r_{\infty}$.\newline
To recapitulate, we get an increasing subsequence $(\phi\circ\psi(n))_{n\geq0}$ such that we have:
$$\lim\limits_{n \rightarrow +\infty} K_{\phi\circ\psi(n)}.A_{\phi\circ\psi(n)}=
\begin{pmatrix} 
\cos(\alpha_{\infty}) & -\sin(\alpha_{\infty}) \\
\sin(\alpha_{\infty}) & \cos(\alpha_{\infty}) 
\end{pmatrix}
.
\begin{pmatrix} 
r_{\infty} & 0 \\
0 & 1/r_{\infty} 
\end{pmatrix}
.$$

Therefore, the $SL(2,\mathbb{R})$-orbit of $(X,q)$ is closed if and only if its $H$-orbit is closed where $H=\left\lbrace 
\begin{pmatrix} 
1 & x \\
0 & 1 
\end{pmatrix}
,
x \in
\mathbb{R}
\right\rbrace 
$.\newline

If $core(X)$ decomposes into finite volume cylinders with commensurable moduli, then the Veech group of $(X,q)$ is of cyclic parabolic type (Proposition 4.2) and the $H$-orbit of $(X,q)$ is closed. Then, its $SL(2,\mathbb{R})$-orbit is closed too.\newline

The last case is about surfaces $(X,q)$ such that all saddle connections of $\partial\mathcal{C}(X)$ share the same direction (without loss of generality, we suppose this is the horizontal direction) and $core(X)$ does not decompose into a family of cylinders of commensurable moduli. Either there is a minimal component or two cylinders whose moduli are not commensurable. Following Proposition 4.2, the Veech group of such a surface $(X,q)$ is not of cyclic parabolic type and its $H$-orbit is not periodic.\newline
We denote by $\mathcal{I}\mathcal{C}(X)$ the interior of $core(X)$. Since the boundary of $\mathcal{U}$ is made of horizontal saddle connections whose total holonomy is zero, we can glue the boundary components on each other while adding marked points if necessary or modifying the order of the conical singularities. We get a flat surface $(X_{c},q_{c})$ without poles of higher order but with some marked horizontal saddle connections (those that have been glued on each other). Flat surface $(X_{c},q_{c})$ belongs to a stratum $\mathcal{Q}$ of meromorphic quadratic differentials with at most simple poles.\newline
No saddle connection shrinks under the action of $H$. Thus, lengths of saddle connections remain bounded below by a positive number all along the orbit. Therefore $H.(X_{c},q_{c})$ lies in a compact subset of $\mathcal{Q}$. So if $H.(X_{c},q_{c})$ is closed, then it is also compact and consquently $H.(X_{c},q_{c})$ is a periodic orbit. However, as the action of $H$ preserve the marked horizontal saddle connections, the fact that $H.(X,q)$ is not periodic implies that $H.(X_{c},q_{c})$ is not periodic either. Therefore, $H.(X_{c},q_{c})$ is not closed. As the action of $H$ preserve the marked horizontal saddle connections, the $H$-orbit of $(X,q)$ is not closed either. Consequently, the $SL(2,\mathbb{R})$-orbit of $(X,q)$ is not closed in this case. This ends the proof.\newline
\end{proof}

\section{Generic Veech groups}

Just as flat surfaces of finite area, the typical flat surface with poles of higher order has a trivial Veech group. Theorem 2.4 provides a similar result to what Möller proved in \cite{Mo}. As a result of our definition of a Veech group, we do not have to deal with hyperellipticity.\newline

\begin{proof}[Proof of Theorem 2.4]
Strata of dimension one are strata of flat surfaces of genus zero with $n+p=3$ singularities. Automorphisms of the sphere allow to fix three points so there is a unique surface up to scaling. All of these flat surfaces have a unique saddle connection and therefore their core is degenerated. Following proof of Theorem 2.1, their Veech group is of continuous type.\newline

In strata of dimension at least three, for any flat surface $(X,q)$, we consider the $\mathbb{Z}$-module $A$ in $\mathbb{C}$ generated by the holonomy vectors of the relative homology classes. For a generic choice of periods of the relative homology, there is no matrix of $GL^{+}(2,\mathbb{R})$ that preserves $A$. Therefore, the Veech group of the generic surface in such strata is trivial.\newline

In strata of dimension two, flat surfaces whose Veech group is of continuous type lie a real codimension one subset of the stratum (holonomy vectors of all relative homology classes are $\mathbb{R}$-colinear). Therefore, outside this locus, the Veech group of any surface is discrete.\newline

If the Veech group of a flat surface is of cyclic parabolic type, then following Proposition 4.2, the core decomposes into finite volume cylinders. If there are more than one cylinder in this decomposition, there are more than two complex parameters of deformation of the flat metric. So the core of such a flat surface is formed by a unique cylinder and maybe other saddle connections.\newline
Following Theorem 3.7, the maximal number $|A|$ of noncrossing saddle connections is $6-\beta$. It is also equal to $2+t$. Since a cylinder is formed by at least two ideal triangles, the number $\beta$ of saddle connections in the boundary is at most $2$. If there were other saddle connections in the core other than those that belong to the cylider or its boundary, their two sides would belong to domains of poles and we would have $\beta \geq 3$. Consequently, in flat surfaces whose Veech group is of cyclic parabolic type, the whole core is a cylinder. Such a shape is shared by every flat surface of the same chamber, see Subsection 3.7.\newline
Strata where such chambers exist has already be characterized in \cite{Ta}. These chambers appear in any stratum of the form $\mathcal{Q}(a,-a)$. When the genus is zero, the closed geodesics of a cylinder cuts out the surface into two connected components where the degrees of the singularities sum to $-2$. Therefore, such chambers exist in strata $\mathcal{Q}(a,-1^{2},-2-a)$ with $a\geq1$ and $\mathcal{Q}(a,b,-a-2,-b-2)$ with $a,b\geq1$.\newline
Outside the previous chambers and the locus where the Veech group of the surfaces is of continuous type, surfaces have a finite Veech group. The Veech group preserves the boundary of the core and its action is that of a finite rotation group. In particular, its action does not preserve any direction. If the Veech group of a flat surface is not trivial, then the saddle connections of boundary of the core cannot belong to the same direction. Since the maximal number $|A|$ of noncrossing saddle connections is $6-\beta$, the boundary of the core is formed by at most three saddle connections. There are two loci where the Veech group may be nontrivial:\newline
(i) The core is an ideal triangle and its boundary is formed by three saddle connections.\newline
(ii) The core is degenerated and formed by two saddle connections.\newline 
In particular, if the core is not degenerated and bounded by two saddle connections, the core is formed by two triangles and is a cylinder. This case has already been settled.\newline
These loci are chambers in the stratum because they are characterized by the topological shape of the core.\newline

In the case (i), the only nontrivial action for the Veech group is to send each of the three directions of thee saddle connections of the triangle to the next one in the cyclic order. In this case, the action is conjugated to the action of the rotations of order four. The Veech group in $GL^{+}(2,\mathbb{R})/ \lbrace\pm Id\rbrace$ is therefore isomorphic to $\mathbb{Z}/3\mathbb{Z}$.\newline

In strata where $g=0$ and $p=3$, a nontrivial action should act faithfully on the domains of the poles. Therefore, they should have the same degree. In such strata $\mathcal{Q}(3b-4,-b,-b,-b)$ with $b\geq2$, there is a unique chamber formed by surfaces where the three domains of poles (topological disks) are glued on the three sides of an ideal triangle.\newline

In strata where $g=0$ and $n=p=2$, two of the three saddle connections would lie between two distinct conical singularities while the other would be closed. Therefore, the Veech group could not act nontrivially on the sides of the ideal triangle.\newline

In strata where $g=0$ and $n=3$, the three vertices of the triangle would be distinct conical singularities. The Veech group would act faithfully on them so they should have the same degree. In such strata $\mathcal{Q}(a,a,a,-3a-4)$ with $a\geq 1$, there is a unique chamber formed by surfaces where the core is a triangle. Indeed, up to an action of $GL^{+}(2,\mathbb{R})$, we can assume the triangle is equilateral. The angles of the domain of poles have the same magnitude $\dfrac{(5+3a)\pi}{3}$.\newline

In strata $\mathcal{Q}(a,-a)$ with $a\geq2$, we consider a flat surface where the core is a triangle. There are three angles in the triangle and three angles in the domain of the pole. For every pair of sides of the triangle, there is a loop that leave the triangle by the first side and then return to the triangle by the second side. There are three such elementary loops. If there is a nontrivial action of the Veech group, then these elementary loops have the same topological degree $\dfrac{k}{2}$ where $k$ is an integer number. A direct computation of the angles concerned by the loops show that the total angle in the surface is $2\pi +3k\pi$. Therefore, such a surface is in $\mathcal{Q}(3k,-3k)$ with $k\geq 1$. Since topological indexes and the shape of the core are the same for every surface of a given chamber. Flat surfaces of $\mathcal{Q}(3k,-3k)$ where the Veech group is isomorphic to $\mathbb{Z}/3\mathbb{Z}$ are exactly those whose core is a triangle and whose elementary loops have the same degree.\newline

In the case (ii), the only nontrivial action for the Veech group is to interchange the two directions of the saddle connections. In this case, the action is conjugated to the action of the rotations of order four. The Veech group in $GL^{+}(2,\mathbb{R})/ \lbrace\pm Id\rbrace$ is therefore isomorphic to $\mathbb{Z}/2\mathbb{Z}$. In a bidimensional stratum, if the Veech group of a flat surface $(X,q)$ is not of continuous type, the $GL^{+}(2,\mathbb{R})$-orbit of $(X,q)$ is an open subset of the stratum. Since the Veech group is invariant in a $GL^{+}(2,\mathbb{R})$-orbit, we can assume saddle connections have the same length and the magnitude of every angle is a multiple of $\dfrac{\pi}{2}$.\newline

In strata where $g=0$ and $n=p=2$, there are two shapes of degenerate core. In the first case, there is a closed saddle connection and one saddle connection whose ends are distinct conical singularities. The Veech group should preserve the closed saddle connection and its action could not be conjugated to rotations of order four. Therefore, we assume surfaces are formed by two domains of poles separated by a pair of saddle connections ending in distinct conical singularities. Residues at the poles of even degree are globally preserved by the action of the Veech group so we can assume poles are of odd degree. If the poles have not the same degree, then every domain of pole is preserved by rotations of order four and thus the two angles of each domain of pole are congruent. Consequently, the two conical singularities have the same degree. It is clear that the flat surface of $\mathcal{Q}(k+l-1,k+l-1,-2k-1,-2l-1)$ with $k,l \geq 1$ whose angles are $\dfrac{(1+2k)\pi}{2}$, $\dfrac{(1+2k)\pi}{2}$, $\dfrac{(1+2l)\pi}{2}$ and $\dfrac{(1+2l)\pi}{2}$ has a Veech group isomorphic to $\mathbb{Z}/2\mathbb{Z}$.\newline
If on the contrary, the two poles have the same degree whereas the two conical singularities have a different degree, then the action of the Veech group permutates the two angles of each conical singularity. It is possible only if the total angle of each conical singularity is an odd multiple of $\pi$. It is clear that the flat surface of $\mathcal{Q}(2k-1,2l-1,-k-l-1,-k-l-1)$ with $k,l \geq 1$ whose angles are $\dfrac{(1+2k)\pi}{2}$, $\dfrac{(1+2k)\pi}{2}$, $\dfrac{(1+2l)\pi}{2}$ and $\dfrac{(1+2l)\pi}{2}$ has a Veech group isomorphic to $\mathbb{Z}/2\mathbb{Z}$.\newline

In strata where $g=0$ and $n=1$, there is a central domain of pole and two peripherical domains of poles. If the Veech group acts by rotations of order four, the directions of the two saddle connections are permuted so the two peripherical poles have same degree. Similarly, the two angular sectors in the central domain of pole should be congruent and of the form $\dfrac{(2k+1)\pi}{2}$ with $k \geq 1$. Such a flat surface belongs to a stratum $\mathcal{Q}(2b+2k-3,-b,-b,-2k-1)$ with $b\geq2$ and $k\geq1$. It is clear that such a flat surface is invariant by rotations of order four.\newline

In strata where $g=0$ and $n=3$, there is a central conical singularity from which start the two saddle connections ending in peripherical conical singularities. There are two angular sectors around the central conical singularity and rotations of order four should transform each into the other. Thus they have the same magnitude $\dfrac{(2k+1)\pi}{2}$ with $k \geq 1$. The central conical singularity is of degree $2k-1$. Similarly, the peripherical conical singularities should have the same degree $a$. Consequently, in strata where $g=0$ and $n=3$, flat surfaces with a Veech group isomorphic to $\mathbb{Z}/3\mathbb{Z}$ are those that belong to the unique $GL^{+}(2,\mathbb{R})$-orbit of a surface this such precribed angles in strata $\mathcal{Q}(2k-1,a,a,-2a-2k-3)$ with $k\geq1$ and $a \in \lbrace-1\rbrace \cup \mathbb{N}^{\ast}$.\newline

In strata $\mathcal{Q}(a,-a)$ with $a\geq2$, there are four angles in the flat surface. Since they are angles between distinct saddle connections, their magnitude is of the form $\dfrac{(2k+1)\pi}{2}$. If the Veech group acts by rotation of order four, then these angles have the same magnitude and the total angle in the flat surface is $(4k+2)\pi$. Therefore, the Veech group is isomorphic to $\mathbb{Z}/2\mathbb{Z}$ if and only if the flat surface belongs to the unique open $GL^{+}(2,\mathbb{R})$-orbit of the surface with four congruent angles in $\mathcal{Q}(4k,-4k)$ with $k\geq1$. Such flat surfaces have trivial holonomy.\newline
\end{proof}

\begin{ex}
There is a chamber $\mathcal{C}$ in $\mathcal{Q}(2,-2)$ such that every surface in $\mathcal{C}$ has a Veech group of cyclic parabolic type, see Figure 4.

\begin{figure}
\includegraphics[scale=0.3]{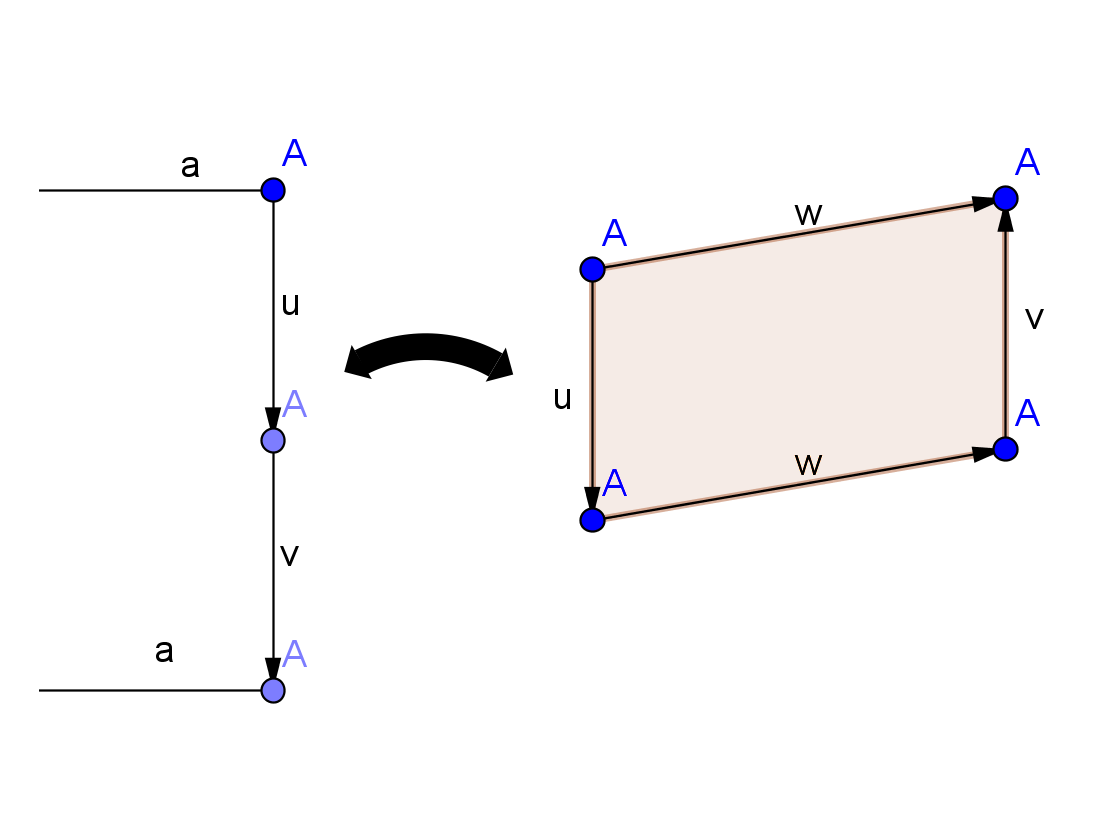}
\caption{A surface in $\mathcal{Q}(2,-2)$ whose Veech group is of cyclic parabolic type.}
\end{figure}

\end{ex}

\nopagebreak
\vskip.5cm

\begin{thebibliography}{10}
  
\bibitem{Bo1}
C. Boissy.
\newblock {\em Connected components of the strata of the moduli space of meromorphic differentials}.
\newblock {Commentarii Mathematici Helvitici}, Volume 90, Issue 2, 255-286,
  2015.
  
\bibitem{BCGGM1}
M. Bainbridge, D. Chen, Q. Gendron, S. Grushevsky, M. Möller.
\newblock {\em Strata of $k$-differentials}.
\newblock {Preprint, arXiv:1610.09238}
  
\bibitem{EM}
J. Ellenberg, D.B. McReynolds.
\newblock {\em Arithmetic Veech sublattices of $SL(2,\mathbb{R})$}.
\newblock {Duke Mathematical Journal}, Volume 161, Number 3, 415-429,
  2012.
  
\bibitem{HKK}
F. Haiden, L. Katzarkov, M. Kontsevich.
\newblock {\em Flat surfaces and stability structures}.
\newblock {Preprint, arXiv:1409.8611},
  2015.

\bibitem{HL}
P. Hubert, E. Lanneau.
\newblock {\em Veech groups without parabolic elements}.
\newblock {Duke Mathematical Journal}, Volume 133, Number 2, 335-346,
  2003.
  
\bibitem{HSc}
P. Hubert, T. Schmidt.
\newblock {\em An introduction to Veech Surfaces}.
\newblock {Handbook of Modern Dynamics}, Volume 1B,
  2006.
  
\bibitem{HS}
P. Hubert, G. Schmithüsen.
\newblock {\em Infinite translation surfaces with infinitely generated Veech groups}.
\newblock {Journal of Modern Dynamics}, Volume 4, Number 4, 715-732,
  2010.
  
\bibitem{Mo}
M. Möller.
\newblock {\em Affine groups of flat surfaces}.
\newblock {Handbook of Teichmüller theory, 369-387},
  2009.

\bibitem{MZ}
H. Masur, A. Zorich.
\newblock {\em Multiple saddle connections on flat surfaces and the principal boundary of the moduli spaces of quadratic differentials}.
\newblock {Geometric and Functional Analysis}, Volume 18, Issue 3, 919-987,
  2008.
  
\bibitem{Sc}
G. Schmithüsen.
\newblock {\em An algorithm for finding the Veech group of an origami}.
\newblock {Preprint},
  2003.
  
\bibitem{St}
K. Strebel.
\newblock {\em Quadratic Differentials}.
\newblock {Ergebnisse der Mathematik und ihrer Grenzgebiete}, Volume 5,
  1984.

\bibitem{Ta}
G. Tahar.
\newblock {\em Counting saddle connections in flat surfaces with poles of higher order}.
\newblock {Preprint, arXiv:1606.03705},
  2016.

\bibitem{Va}
F. Valdez.
\newblock {\em Veech groups, irrational billards and stable abelian differentials}.
\newblock {Discrete and Continuous Dynamical Systems}, Series A, Volume 32, 1055-1063,
  2012.

\bibitem{Ve}
W. Veech.
\newblock {\em Teichmüller curves in modular space, Eisenstein series, and an application to triangular billiards}.
\newblock {Inventiones Mathematicae}, Volume 97, 553-583,
  1989.

\bibitem{Zo}
A. Zorich.
\newblock {\em Flat Surfaces}.
\newblock {Frontiers in Physics, Number Theory and Geometry, 439-586},
  2006.

\end{thebibliography}
\end{document}